\documentclass[12pt]{amsart}
\usepackage{a4}
\usepackage{amssymb}
\usepackage{amsmath}
\usepackage{amsthm}
\usepackage{amstext}
\usepackage{amscd}
\usepackage{latexsym}
\usepackage{graphics}
\usepackage{color}
\usepackage{url}
\swapnumbers
\theoremstyle{plain}
\newtheorem{thm}{Theorem}[section]
\newtheorem{q}[thm]{Question}
\newtheorem{lemma}[thm]{Lemma}

\newtheorem{corollary}[thm]{Corollary}
\newtheorem{prop}[thm]{Proposition}

\theoremstyle{definition}

\newtheorem{rmk}[thm]{Remark}

\newtheorem{Construction}[thm]{Construction}
\newtheorem{defn}[thm]{Definition}

\newtheorem{claim}[thm]{Claim}

\newcommand{\Hom}{{\rm Hom}}

\newcommand{\Ext}{{\rm Ext}}
\newcommand{\Spec}{{\rm Spec \,}}

\newcommand{\Sch}{{\rm Sch}}

\newcommand{\NE}{{\rm NE}}

\renewcommand{\tilde}{\widetilde}
\newcommand{\sM}{{\mathcal M}}

\newcommand{\sS}{{\mathcal S}}

\newcommand{\A}{{\mathbb A}}

\newcommand{\F}{{\mathbb F}}

\renewcommand{\P}{{\mathbb P}}
\newcommand{\Q}{{\mathbb Q}}

\newcommand{\Z}{{\mathbb Z}}

\input{xy}
\xyoption{all}

\begin{document}
\title{Strong Rational Connectedness of Surfaces}
\author{Chenyang Xu}

\begin{abstract}
This paper focuses on the study of the strong rational connectedness
of smooth rationally connected surfaces. In particular, we show that
the smooth locus of a log del Pezzo surface is strongly rationally
connected. This confirms a conjecture due to Hassett and Tschinkel
in \cite{ht08}.
\end{abstract}

\maketitle

%%%%%%%%%%%%%%%%%%%%%%%%%%%%%
\tableofcontents

%%%%%%%%%%%%%%%%%%%%%%%%%%%%%%%%%%%%%%%%%%%%%%%%%%%%%%%%%%%%%%%%%%%%%%%%
\vspace*{6pt}
\section{Introduction}
Throughout this paper, unless otherwise mentioned, the
ground field $k$ will be an algebraically closed field of
characteristic zero.
\begin{defn}\label{defn}
A variety $X$ is called {\it rationally connected} if there is a family of proper rational curves $g : U \to Z$  and a cycle morphism $u : U \to X$ such that $u^{(2)} : U \times_Z U\to X \times_k X$ is dominant.
\end{defn}

We refer to \cite{kollarrc} for the background of rationally connected varieties. Given a smooth variety $X$, it is rationally connected if and only if there is a morphism $f: \P^1 \to X$ such that $f^*T_X$ is ample (cf. \cite{kollarrc}, IV.3.3.1 and 3.7). We call such a curve {\it a very free curve}. Similarly, a curve $f:\P^1 \to X$ is called a {\it free curve} if $f^*T_X$ is semi-ample. It is known (cf. ibid, IV.3.9.4) that for any smooth rationally connected variety $X$ over $k$, there is a unique largest nonempty open subset $X^{vf}\subset X$ with the property that $x\in X^{vf}$ if and only if there is a morphism $f:\P^1 \to X^{vf}$ such that:
\begin{enumerate}
 \item $x\in f(\P^1)$;
\item $f^*T_X$ is ample.
\end{enumerate}

We call $X^{vf}$ the {\it very free locus} of $X$. If $X^{vf}=X$, $X$ is called {\it strongly rationally connected} (cf. \cite{ht08}, 14).
\begin{thm}[\cite{kmm},  2.1; \cite{kollarrc}, IV.3.9]\label{proper} A proper and smooth rationally connected variety $X$ is strongly rationally connected.
\end{thm}

A result of Campana and Koll\'ar-Miyaoka-Mori states
that all smooth Fano varieties are rationally connected (\cite{Ca}, \cite{KMM} or \cite{kollarrc}, V.2.1). More generally, log $\Q$-Fano varieties are rationally connected (\cite{zhang} or \cite{hm07}). A {\it log $\Q$-Fano variety} $(X,\Delta)$ is a normal projective variety $X$ with a Weil $\mathbb{Q}$-divisor $\Delta$ such that $(X, \Delta)$ is $klt$ and $K_X+\Delta$ is anti-ample. A log $\Q$-Fano surface is also called a {\it log del Pezzo surface}. \\

Given a log $\Q$-Fano variety $(X,\Delta)$, one can ask the question whether the smooth locus $X^{sm}$ of the underlying space $X$ is rationally connected. In general determining the rational connectedness of a nonproper variety can be formidable, and the answer to this question is not known except for the surface case.

\begin{thm}[\cite{km99}, 1.6]\label{rc}
 Let $(S,\Delta)$ be a log del Pezzo surface, then the smooth locus $S^{sm}$ of $S$ is rationally connected.
\end{thm}

If we drop the properness assumption, it is an open question whether rational connectedness is equivalent to strong rational connectedness (\cite{kollarrc}, IV.3.9.5). The main purpose of this paper is to show that the smooth loci of log del Pezzo surfaces are not only rationally connected but also strongly rationally connected, i.e., these two notions coincide in this case. This confirms a conjecture due to Hassett and Tschinkel (\cite{ht08}, Conjecture 19).
\begin{thm}\label{main}
 Let $(S,\Delta)$ be a log del Pezzo surface, then the smooth locus $S^{sm}$ is strongly rationally connected.
\end{thm}

Applying the same technique, we can prove the following result.
\begin{thm}\label{duval}
Let $S$ be a projective surface which has at worst Du Val singularities. If $S^{sm}$ is rationally connected, then $S^{sm}$ is indeed strongly rationally connected.
\end{thm}

The strategy to prove (\ref{main}) consists of two
steps: first, using an argument from the Minimal Model Program (MMP),
and applying Keel-M$^{c}$Kernan's result (\ref{rc}), we prove
(\ref{cdm1}) saying that there are at most finitely many points in
$S^{sm}$ which are not in the very free locus. In the second step,
we show that indeed $S^{sm}$ is equal to its very free
locus in the following way: Given a smooth point $y$ of $S$, we choose a
sequence of general points $y_t$ to specialize to it. Applying
Keel-M$^c$Kernan's theorem again, we can connect $y_t$
 with a fixed general point $x\in S$ by a proper rational curve contained in $S^{sm}$. Thus it degenerates to
  a limiting object containing $x$ and $y$. We show
that after adding very free curves to the limiting object, we can
smooth it away from the singular points of $S$ in the way that the smoothing  curves are very free and still pass $y$.

For the purpose of doing deformation, we  study (twisted) rational curves on the
smooth Deligne-Mumford stack $\sS$, which realizes $S$ as its coarse
moduli space. The theory of the moduli space of twisted stable maps
with Deligne-Mumford stacks being the target spaces
(cf. \cite{av}) provides a fine framework. The advantage
of considering $\mathcal{S}$ is that its smoothness allows us to do
the deformation theory in a similar way as the classical case when
the target space is a smooth variety (see Section 2 and 3). We get
the limiting object, which could be reducible and twisted, from the
properness of $\overline{\mathcal{M}}_{0,2}(\mathcal{S},d)$. Since the first step
allows us to attach teeth to general points in each component of the
image of the limiting object, by imitating the analogue argument for
the classical case, we show that after attaching enough teeth, it
eliminates all the obstructions of deforming the morphism with $x$ and $y$ fixed. Therefore, from the the computation on the first order deformation, we conclude that we can smooth 
the limiting object with the additional property that the image still contains $y$. The smoothing curve will be a very free curve, which we can
choose to be entirely contained in the scheme locus of $\mathcal{S}$
(see (\ref{finitepoints})).

For some special log del Pezzo surfaces, the strong
rational connectedness of the smooth loci has been established in
\cite{ht08} and \cite{kn}. These results have been applied to
prove certain weak approximation results for rational surfaces over
the function field of curves. In \cite{xu}, we use the results
of this note to establish weak approximations for more del Pezzo
surfaces.

From our argument, it is also interesting to ask the
following question: if we assume $U$ to be a nonproper smooth
rationally connected surface, and $S\supset U$ a normal
compactification, then what type of singularities can $S$ have? We
give the following answer.

\begin{thm}\label{singularity}
Let $S$ be a normal surface such that the smooth locus $S^{sm}$ is rationally connected. Then  $S$ contains only rational singularities.
\end{thm}

 The content of this note is organized as follows: in
Sections 2 and 3, we develop a twisted version of the classical
theory of morphisms from curves to varieties; then we apply it to show our main theorems (\ref{main}) and (\ref{duval}) in
Section 4;  in
Section 5, we give an example due to Koll\'ar who constructs a  log del Pezzo surface
with $A_1$ singularities in characteristic 2, whose smooth locus
does not contain any free curve.

\noindent {\bf Acknowledgement:} The author would like to thank Igor
Dolgachev, Brendan Hassett, Amit Hogadi, Amanda Knecht, J\'anos
Koll\'ar and Jason Starr for useful conversations and emails. He especially wants to
thank Dan Abramovich for suggesting a crucial idea to prove
Theorem (\ref{main}). Thanks to Garving Luli  for his help on
English and to the referee for enormous helpful suggestions on the exposition; any remaining mistakes are my own. The author was partially
supported by Clay liftoff fellowship. Part of the work was done when
the author visited Universit\"at Duisburg-Essen. The author wishes
to thank H\'el\`ene Esnault for her hospitality during the visit.
This material is also based upon work when the author was in
Institute for Advanced Study and supported by the NSF under
agreement No. DMS-0635607. The author was partially supported by NSF research grant no: 0969495.

\section{Twisted curves and twisted stable maps }
In this section, we give a short introduction to the theory of twisted curves and $n$-pointed twisted stable maps. We refer to \cite{av} and \cite{ol} for more details.

Let $S$ be a scheme and $\mathcal{C}/S$ a proper flat  Deligne-Mumford stack  whose fibers are
purely one-dimensional and geometrically connected with at most nodal singularities. Let $\mathcal{C} \to C$ be the coarse moduli space of $\mathcal{C}$, and let $C^{sm} \subset C$
be the open subset where $C \to S$ is smooth. Assume that the inverse image $\mathcal{C} \times_C C^{sm} \subset \mathcal{C}$ is equal
to the open substack of $\mathcal{C}$ where $\mathcal{C} \to S$ is smooth and that for every geometric point $\bar{s} \to S$ the
map $\mathcal{C}_{\bar{s}} \to C_{\bar{s}}$ is an isomorphism over some dense open subset of $C_{\bar{s}}$. Then the coarse space $C$ is a
nodal curve over $S$, and it is well known that (cf. \cite{ol}, 2.2) for any geometric point mapping to
a node $\bar{s} \to C$, there exists an \'etale neighborhood $\Spec(A) \to C$ of $\bar{s}$ and an \'etale morphism
$$\Spec(A) \to \Spec_S(\mathcal{O}_S[x, y]/xy-t) $$
for some $t \in \mathcal{O}_S$, such that the pullback $\mathcal{C}
\times_C \Spec(A)$ is isomorphic to
$$[\Spec(A[z,w]/zw = t', z^n = x,w^n = y)/\Gamma] $$
for some element $t' \in \mathcal{O}_S$. Here $\Gamma$ is a finite cyclic group of order $n$, such that if
$\gamma \in \Gamma$ is a generator then $\gamma(z) = \xi z$ and $\gamma(w) = \xi'w$ for some primitive $n$-$th$ roots of unity $\xi$ and $\xi'$.
The stack is called {\it balanced} if \'etale locally there exists such a description with $\xi' = \xi^{-1}$.

\begin{defn}[\cite{av}, 4.1.2] A {\it twisted curve} is a stack $\mathcal{C} \to S$ as above such that the action at
each nodal point $\bar{s}\to \mathcal{C} $  is balanced. A twisted curve $\mathcal{C}\to S$ has genus $g$ if the genus of $C_{\bar{s}}$ is $g$ for
every geometric point $\bar{s} \to S$. An {\it $n$-pointed twisted curve} is a twisted curve $\mathcal{C} \to  S$ together with a
collection of disjoint closed substacks $\{\Sigma_i\}^n _{i=1}$ of $\mathcal{C}$ such that:
\begin{enumerate}
\item each $\Sigma_i \subset \mathcal{C}$ is contained in the smooth locus of $\mathcal{C} \to S$;
\item the stacks $\Sigma_i$ are \'etale gerbes over $S$;
\item if $\mathcal{C}_{gen}$ denotes the complement of the $\Sigma_i$ in the smooth locus of $\mathcal{C} \to S$, then $\mathcal{C}_{gen}$ is a scheme.
\end{enumerate}
\end{defn}

\begin{rmk}
Unlike \cite{av}, we will only use balanced twisted curves in this paper, so we usually omit the adjective ``balanced''. We note that Abramovich and Vistoli prove that the moduli stack parametrizing the twisted stable maps from balanced twisted curves is an open and closed substack of the stack parametrizing all twisted stable maps (cf. \cite{av} 8.1.1).
\end{rmk}

In \cite{ol}, Olsson obtains the following description of the versal deformation of a given $n$-pointed twisted curve. Let $\mathcal{M}^{tw}_{g,n}$ denote the fibered category over $\Z$ which to any scheme $T$ associates the
groupoid of $n$-marked genus $g$ twisted curves $(\mathcal{C}, {\Sigma_i})$ over $T$, and let $\mathcal{S}_{g,n}$ denote the fibered category over $\Z$ which to any scheme $S$ associates the groupoid of all (not necessarily
stable) $n$-pointed genus $g$ nodal curves $C/S$. Now consider a field $k$ and an object $(\mathcal{C}, {\Sigma_i}) \in \mathcal{M}^{tw}_{g,n}(k)$. Let $(C, {\sigma_i})$ be its coarse
moduli space, and let $R$ be a versal deformation space for the object $(C, {\sigma_i}) \in \mathcal{S}_{g,n}(k)$. Let
$q_1, . . . , q_m \in C$ be the nodes,  $r_i$ the order of the stabilizer group of a point of $\mathcal{C}$ lying
above $q_i$. It is well-known (cf. \cite{dm}) that there is a smooth divisor $D_i \subset \Spec(R)$ classifying deformations, where
$q_i$ remains a node. In other words, if $t_i \in R$ is an element defining $D_i$ then in an \'etale neighborhood
of $q_i$ the versal deformation $\tilde{C} \to \Spec(R)$ of $(C, {\sigma_i})$ is isomorphic to
$$\Spec(R[x, y]/xy - t_i).$$ As a corollary to the above result, Olsson shows the following.
\begin{prop}[\cite{ol}, 1.10 and 1.11]\label{versal}
(Notations as above.) A versal deformation space for the twisted
curve $(\mathcal{C}, {\Sigma_i})$ is given by
$$R[z_1, . . . , z_m]/(z^{r_1}_1 - t_1, . . . , z^{r_m}_m -t_m).$$
\end{prop}

 In our later discussion, we will be interested in the smooth objects of the versal family. From this proposition, we know that such objects always exist. \\ 

 Now we consider a proper tame Deligne-Mumford stack $\mathcal{M}$ admitting a projective coarse moduli scheme $M$. We fix an ample line bundle on $M$. To compactify the moduli space of $n$-pointed stable maps from curves to a Deligne-Mumford stack, Abramovich and Vistoli made the important observation that we have to enlarge the source from the category of curves to the category of twisted curves. We remark that the collection of twisted curves over a scheme $S$ form a 2-category which is in fact equivalent to a 1-category (cf. \cite{av}, 4.4.2).

\begin{defn}
{\it A twisted stable $n$-pointed map of genus $g$ and degree $d$ over $S$}
$(\mathcal{C} \to S,  f : \mathcal{C} \to \mathcal{M},\Sigma^{\mathcal{C}}_i \subset \mathcal{C})$
consists of a commutative diagram
$$\xymatrix{
\mathcal{C} \ar@{->}^f[r]\ar[d] & \mathcal{M}\ar[d] \\
C           \ar@{->}^{\overline{f}}[r]\ar[d]     & M               \\
S
}$$
along with $n$ closed substacks $\Sigma^{\mathcal{C}}_i\subset \mathcal{C}$, satisfying:
\begin{enumerate}
 \item  $\mathcal{C} \to C \to S$ along with $\Sigma^{\mathcal{C}}_i$
is a twisted nodal $n$-pointed curve over $S$;
\item the morphism $f: \mathcal{C} \to \mathcal{M}$ is representable; and
\item The induced structure on the coarse moduli spaces $$(C \to S,\overline{f} : C \to M,\sigma_i)$$ yields an (untwisted) stable $n$-pointed map of degree $d$.
\end{enumerate}
\end{defn}

\begin{thm}[\cite{av}, 1.4.1] Let $\overline{\mathcal{M}}_{g,n}(\mathcal{M},d)$ be fibered over $\Sch/S$, the category of the twisted stable $n$-pointed maps $\mathcal{C} \to \mathcal{M}$ of genus $g$ and degree $d$.
\begin{enumerate}
 \item The category $\overline{\mathcal{M}}_{g,n}(\mathcal{M},d)$ is a proper Deligne-Mumford stack.
\item The coarse moduli space $\overline{M}_{g,n}(\mathcal{M},d)$ of $\overline{\mathcal{M}}_{g,n}(\mathcal{M},d)$ is projective.
\item There is a commutative diagram
$$\xymatrix{
\overline{\mathcal{M}}_{g,n}(\mathcal{M},d) \ar@{->}[r]\ar[d] & \overline{\mathcal{M}}_{g,n}(M,d) \ar[d] \\
\overline{M}_{g,n}(\mathcal{M},d)         \ar@{->}[r]     & \overline{M}_{g,n}(M,d),
}$$
where the top arrow is proper, quasi-finite and relatively of Deligne-Mumford type, and the bottom
arrow is finite. 
\end{enumerate}

\end{thm}

In our paper, we also need to study morphisms between stacks and their deformation theory, which has been investigated in \cite{ol2} and \cite{ol3}. From Olsson's results, we can see that when the stacks are Deligne-Mumford, the theory is similar to the case of schemes. We summarize all the results which we will need.
\begin{thm}\label{rep}
 Let $\mathcal{C}/S$ be a twisted curve and $\mathcal{Y}/S$ a Deligne-Mumford stack.
  $\underline{\Hom}_S(\mathcal{C},\mathcal{Y})$ is represented by a Deligne-Mumford stack. The fibered subcategory of representable morphisms
$$\underline{\Hom}^{rep}_S(\mathcal{C},\mathcal{Y}) \subset \underline{\Hom}_S(\mathcal{C},\mathcal{Y})$$
is an open substack.
\item
\end{thm}
\begin{proof}
See the proof of (\cite{ol}, 1.15).
\end{proof}

\begin{thm}\label{dfmthy}
Let the notations be as \eqref{rep}. Choose $S=\Spec(A)$,
where $A$ is a ring.  Assume
$a:\mathcal{C}/A\to \mathcal{Y}/A$ a representable morphism.  Denote
$L_{\mathcal{Y}/A}$ as the relative cotangent complex. If
$\gamma:\Spec(A)\to \Spec(A')$ is a closed immersion defined by a
square-zero ideal $M\subset A'$, and $i:\mathcal{C}/A\to
\mathcal{C'}/A'$ is a flat extension with an isomorphism
$\mathcal{C'}\times_{A'} A \cong \mathcal{C}$ and $j:\mathcal{Y}'/A'\to \mathcal{Y}/A$ a flat extension with $\mathcal{Y}'\times_{A} A'\cong \mathcal{Y}$, then
\begin{enumerate}
 \item There is a canonical obstruction class $o(a,i,j)\in \Ext^1(La^*L_{\mathcal{Y}/A}, M)$ whose vanishing is necessary and sufficient for the existence of an arrow $a':\mathcal{C'}\to \mathcal{Y}$ extending $a:\mathcal{C}\to \mathcal{Y}$.
\item If $o(a,i,j)=0$, then the set of isomorphism classes of maps $a':\mathcal{C'}\to \mathcal{Y}$ as in (1) is naturally a torsor under $\Ext^0(La^*L_{\mathcal{Y}/A},M)$.
\end{enumerate}

\end{thm}

\begin{proof}
It follows from (\cite{ol2},  1.5 (i) and (ii)), in which we choose  $Z= {\rm Spec}(A)$ and $Z'={\rm Spec}(A')$. 
\end{proof}

\begin{rmk}\label{obstruction}
Theorem (\ref{dfmthy}) can be used to define an obstruction theory in the sense
of (\cite{ar}, 2.6) for the stack $\underline{\Hom}^{rep}_S(\mathcal{C},\mathcal{Y}) $. Recall that such a theory consists of the data (in the following all rings are over $S$):

\begin{enumerate}
 \item  For every surjection of noetherian rings $A \to A_0$ with nilpotent kernel
and $a \in \underline{\Hom}^{rep}_S(\mathcal{C},\mathcal{Y})(A)$, a functor
$$O_a : (A_0 \mbox{-}{\it modules\  of \ finite\  type}) \to (A_0\mbox{-}{\it modules\  of\  finite\  type}).$$

\item For each surjection $A' \to A$ with kernel $M$, an $A_0$-module of finite
type, a class $o_a \in O_a(M)$ which is zero if and only if there exists a lifting of $a$ to
$A'$.
\end{enumerate}

This data is further required to be functorial and linear in $(A_0, M)$. Denote by $\mathcal{Y}^A$ (resp. $\mathcal{C}^A$) the fiber product $\mathcal{Y}\times_S {\rm Spec}(A)$ (resp. $\mathcal{C}\times_S {\rm Spec }(A)$). Now for the morphism
$$a \in \underline{\Hom}^{rep}_S(\mathcal{C},\mathcal{Y})(A)\cong {\Hom}^{rep}_A(\mathcal{C}^A,\mathcal{Y}^A),$$ 
we apply \eqref{dfmthy} to it with $\mathcal{C}'=\mathcal{C}\times_S {\rm Spec}(A')$ and $\mathcal{Y}'=\mathcal{Y}\times_S {\rm Spec}(A')$.
We obtain an obstruction theory  by defining 
$$O_a(M):=\Ext^1(La^*L_{\mathcal{Y}^A/{A}},M),$$ and
taking for each $A' \to A$ as in (\ref{obstruction}.2) inducing $i:\mathcal{C}^A\to \mathcal{C'}$ and $j:\mathcal{Y}^A\to \mathcal{Y}'$, the class 
$$o(a,i,j) \in \Ext^1(La^*L_{\mathcal{Y}^A/A}, M)$$ to be the
class obtained from (\ref{dfmthy}.1). Note that it follows from the construction of the
cotangent complex  that the homology groups of $La^*L_{\mathcal{Y}^A/A}$ are
coherent. From this and standard properties of cohomology it follows that the $A_0$-modules $\Ext^1(La^*L_{\mathcal{Y}^A/A}, M)$ are of finite type and that the additional conditions
(\cite{ar}, 4.1) on the obstruction theory hold. See (\cite{ol2}, 1.7) for a similar discussion.
\end{rmk}

A corollary to the above discussion is the following analogy of (\cite{kollarrc},  I.2.17):
\begin{corollary}\label{dimension}
Let $\mathcal{F}:\mathcal{C}/S\to \mathcal{Y}/S$ be a representable
proper morphism from a twisted curve to a smooth Deligne-Mumford
stack $\mathcal{Y}/S$. Let $F$ be a field and $s: \Spec F \to S$ a
morphism. Let $\mathcal{C}^F$ and $\mathcal{Y}^F$ be the fibers over
$\Spec(F)$ and $f:\mathcal{C}^F\to \mathcal{Y}^F$ the morphism
induced by $\mathcal{F}$. Assume that $S$ is equi-dimensional at $s$.
Then the dimension of every irreducible component of
$\underline{\Hom}_S(\mathcal{C}, \mathcal{Y})$  at $[f]$ is at least
$$H^0(\mathcal{C},f^*T_{\mathcal{Y}^{F}})-H^1(\mathcal{C},f^*T_{\mathcal{Y}^{F}})+\dim_sS.$$
Furthermore, if $H^1(\mathcal{C},f^*T_{\mathcal{Y}^{F}})=0$, then $\underline{\Hom}_S(\mathcal{C},\mathcal{Y})$ is smooth over $S$ at $[f]$.
\end{corollary}

\begin{proof}
 Because of (\ref{dfmthy}.2), we know that the tangent space of $\underline{\Hom}_s(\mathcal{C}^F,\mathcal{Y}^F)$ (the fiber of $\underline{\Hom}_S(\mathcal{C},\mathcal{Y})$ at $s$) is isomorphic to  $H^0(\mathcal{C},f^*T_{\mathcal{Y}^F})$. Let $(Q,m_Q)$ be the local ring of a point of an \'etale altas of $\underline{\Hom}_S(\mathcal{C},\mathcal{Y})$, whose image is $[f]$. It is a quotient of a local $\mathcal{O}_{s,S}$-algebra  $R$ with the maximal ideal $m_R$, which is smooth over $S$ of relative dimension $h^0(\mathcal{C},f^*T_{\mathcal{Y}^F})$. Let $K$ be the kernel. Then $K\subset (m_sR,m_R^2)$, where $m_s$ is the maximal ideal of $\mathcal{O}_{s,S}$, because $Q$ and $R$ have the same relative tangent space.

Let $B=R/m_RK$ and $m_B$ be the corresponding maximal
ideal. The kernel $J$ of $B\to Q$ is isomorphic to $K/m_RK$. Because
of (\ref{obstruction}), we know that $J$ has a largest quotient $J \to
J^m$ such that the image of $$o\in
\Ext^1(f^*\Omega_{\mathcal{Y}^F},J)\cong
H^1(\mathcal{C},f^*T_{\mathcal{Y}^F})\otimes J$$ \noindent in
$H^1(\mathcal{C},f^*T_{\mathcal{Y}^F})\otimes J^m$ is 0.  Since
the obstruction theory is linear, we have $$\dim J^m \ge \dim
J-h^1(\mathcal{C},f^*T_{\mathcal{Y}^F}).$$

 On the other hand, if we consider the exact sequence $$0\to I^m \to J \to J^m\to 0,$$
where $I^m$ is defined to be the kernel. Thus it follows from the definition of $J^m$ that there is a split of the surjection $B/I^m\to B/J\cong Q$, which then implies $J^m=0$ (cf. \cite{kollarrc}, I.2.10.6). Thus
 $$\mbox{number of generators of } K=\dim K/m_R K=\dim J\le h^1(\mathcal{C},f^*T_{\mathcal{Y}^F})$$

\end{proof}

\section{Smoothing}
 Henceforth, we focus on the $n$-pointed twisted stable maps of genus 0, i.e., objects in $\overline{\mathcal{M}}_{0,n}(\mathcal{M},d)$. In the classical case, to solve the problems relating to rational curves, one of the main tools is the comb construction and its deformations (cf. \cite{kmm}, \cite{kollarrc} and \cite{ghs}). Here we generalize this construction to $n$-pointed twisted stable maps, and also to the case that the handle is reducible. As mentioned in  Section 1, the key point is that to eliminate the local obstructions, we have to add teeth and deform the entire degenerated twisted curve (instead of deforming only one component as in \cite{kmm}).

\begin{defn}\label{comb}
{\it A comb with $m$ teeth} is a twisted nodal curve  $\mathcal{T}$ over $k$ with the following properties:
\begin{enumerate}
 \item $\mathcal{T}=\mathcal{D} \cup_{1\le i\le m} C^i$, where each $C^i$ is isomorphic to $\P^1$;
\item Every $C^i$ meets $\mathcal{D}$ at a single point $x_i$ which is in the smooth scheme locus of $\mathcal{D}$, and $x_i\neq x_j$ for $i\neq j$; 
\item $C^i \cap C^j =\emptyset$.

\end{enumerate}

We call $\mathcal{D}$ the {\it handle} of $\mathcal{T}$ and $C^i$ the {\it teeth} of $\mathcal{T}$.

\end{defn}
\begin{Construction}\label{construction}
 Let $( f:\mathcal{C}\to \mathcal{M}, \Sigma_i\subset \mathcal{C})$ be an $n$-pointed genus 0 twisted stable map of degree $d(>0)$ over $k$. Now we assume
\begin{enumerate}
\item $\mathcal{M}$ is a smooth Deligne-Mumford stack;
\item there is a strongly rationally connected dense open subscheme $M^0\subset \mathcal{M}$;
\item if $\mathcal{D}_j$ is an irreducible component of $\mathcal{C}$ such that $f(\mathcal{D}_j)$ is not a point in $\mathcal{M}$, then $f(\mathcal{D}_j) \cap M^0\not=\emptyset$. 
\end{enumerate}

Let $\mathcal{D}_1$, $\mathcal{D}_2$,..., $\mathcal{D}_s$ be all the components of $\mathcal{C}$, which are not mapped to $\mathcal{M}$ constantly. Because of assumption (3), on each $\mathcal{D}_j$, for any $m_j\in \mathbb{N}$, we can choose $m_j$ general (scheme) points $p^j_1$,..., $p^j_{m_j}$ mapped into $M^0$.

By assumption (2), we can assume that there exist $m=\sum_{j=1}^s m_j$  $1$-pointed very free curves
$$(g^j_k:\P^1\to M^0, r^j_k),\qquad\mbox{such that}\qquad  g^j_k(r^j_k)=f( p^j_k), \mbox{ for } 1\le j\le s,1\le k\le m_j.$$

Let $W^j_k\subset p^j_k\times \A^m$ be defined by the equation $y^j_k=0$ where the coordinates of $\A^m$ are $(y_1^1,...,y^1_{m_1},...,y^s_{m_s})$. Then $W^j_k$'s are disjoint codimension two substacks  of $\mathcal{C}\times \A^m$, which are indeed schemes. Let $h: \mathcal{S}\to\mathcal{C}\times \A^m\to  \A^m$ be obtained by blowing up all $W^j_k$. Gluing the morphisms $f$ and $g^j_k$, we get a morphism $F:\mathcal{S}_0=h^{-1}(\{0\}) \to \mathcal{M}.$  Each $s_i:\Sigma_i  \to \mathcal{C} $ can be lifted to a natural morphism $\Sigma_i\times \A^m \to \mathcal{S}$ because $\mathcal{S} \to \mathcal{C}\times \A^m$ induces an isomorphism over the images of $\Sigma_i\times \A^m$.

\end{Construction}

\begin{prop}\label{smoothing}
 (Notations as above.) There are a morphism $\sigma:Z\to  \A^m$ from  a curve germ $Z$,  with a commutative diagram
$$\xymatrix{
\mathcal{T} \ar@{->}[r]^{\gamma}\ar[d]_{h_Z} & \mathcal{S} \ar[d]_{h}  \\
Z         \ar@{->}[r]^{\sigma}     & \A^m }$$ \noindent where
$\mathcal{T}\subset \mathcal{S}\times_{\mathbb{A}^m}Z$ is a family of subcombs over $Z$ containing 
 $\Sigma_i \times Z \subset
\mathcal{S}\times_{\A^m} Z $ for all $i$, and a representable morphism
$\phi:\mathcal{T}\to \mathcal{M}$ such that:
\begin{enumerate}

\item[(a)]  for each $i$, the image of $\phi(\Sigma_i\times Z)$ in $\sM$ is a point;

\item[(b)] Let  $\eta\in Z$ be a general point, then the restriction $\gamma_{\eta}$ of $\gamma$ to   the  fiber   $\mathcal{T}_{\eta}:=h_Z^{-1}(\eta)$ 
 induces an 1-isomorphic to the birational transform of $\mathcal{C}\times \{\eta\}$ in $ \mathcal{S}_{\eta}:=\mathcal{S}\times_{\mathbb{A}^m}\{\eta\}$;

\item[(c)]  the fiber $\mathcal{T}_0$ of $\mathcal{T}$ over $\sigma^{-1}(\{0\})$ has $m'$ teeth with
$$m'\ge m-h^1(\mathcal{C},f^*T_{\mathcal{M}}(-\sum \Sigma_i)); \mbox{ and}$$
\item[(d)] $\phi|_{\mathcal{T}_0}$ is given by the restriction of $F$.
\end{enumerate}
\end{prop}

\begin{proof}
Let us consider the natural morphism between Deligne-Mumford stacks
$$\underline{\Hom}_{ \A^m}(\mathcal{S},\mathcal{M}\times \A^m) \to \prod_i \underline{\Hom}(\Sigma_i,\mathcal{ M}) .$$
\noindent
$[F]$ gives a point in $\underline{\Hom}_{ \A^m}(\mathcal{S},\mathcal{M}\times \A^m)$. Let
 $\underline{\Hom}_{ \A^m}(\mathcal{S},\mathcal{M}\times \A^m,F|_{\cup \Sigma_i})$
represent the fiber that contains $[F]$.
Let $c$ be the codimension of an irreducible component of the image of a small neighborhood of
$$[F]\in \underline{\Hom}_{ \A^m}(\mathcal{S},\mathcal{M}\times \A^m,F|_{\cup \Sigma_i}) \ \rm{in} \ \A^m .$$
There is an irreducible curve germ $Z$
$$r:(q\in Z)\to ([F]\in \underline{\Hom}_{\A^m}(\mathcal{S},\mathcal{M}\times \A^m,F|_{\cup \Sigma_i}))$$
such that the image of $r(Z)$ in $\A^m$ is contained in at most $c$ coordinate hyperplanes $(y_i=0)$. By base change we obtain a flat family $\mathcal{W}\to Z$, whose general fiber is a comb with at most $c$ teeth (and the handle 1-isomorphic to $\mathcal{C}$). From the construction, after shrinking $Z$ and removing the irreducible components corresponding to these teeth, we get a family of subcombs $\mathcal{T}$ with a representable morphism $\phi:\mathcal{T}\to \mathcal{M}$ satisfying (a), (b) and (d).

To show (c), combining (\ref{dimension}) and the argument of (\cite{kollarrc}, II.7.9), we know that there is a smooth morphism $\Spec(R) \to  \A^m$ such that (the completion of) $\underline{\Hom}_{ \A^m}(\mathcal{S},\mathcal{M}\times \A^m,F|_{\Sigma_i})$ is defined in $\Spec(R)$ by at most
$h^1(\mathcal{C},f^*T_{\mathcal{M}}(-\sum \Sigma_i))$ equations. Thus we conclude that
$c\le h^1(\mathcal{C},f^*T_{\mathcal{M}}(-\sum \Sigma_i)).$
\end{proof}

\begin{lemma}\label{orbicurve}

Let $(\mathcal{K}/k,x_i(r_i))$ be a smooth proper orbi-curve. Let $\pi:\mathcal{K}\to K$ be its coarse moduli and $\mathcal{L}$  a vector bundle on $\mathcal{K}$ of rank $r$. Then $L=\pi_*(\mathcal{L})$ is a vector bundle on $K$ of the same rank, and we have the following exact sequence
$$0\to \pi^*(L)\to \mathcal{L}\to \oplus F_i\to 0.$$
Here $F_i$ are nontrivial sheaves supported on the gerbs $x_i(r_i)$. Furthermore, we have $$H^j(\mathcal{K},\mathcal{L})=H^j(K,L) \mbox{ for } j=0,1.$$
\begin{proof}
See \cite{av}, 2.3.4.
\end{proof}

\end{lemma}
Under the 1-isomorphism induced by $\gamma_{\eta}$, an irreducible component  of $\mathcal{T}$ is mapped to a point by $f$ if and only if it is mapped to point by $\phi|_{\mathcal{T}_{\eta}}$. In addition to the assumptions in (\ref{construction}), now we also assume that $( f:\mathcal{C}\to \mathcal{M}, \Sigma_i)$ satisfies:
\begin{enumerate}
 \item[(4)] if  $ i\not= j$, then the image of $f(\Sigma_i)$ and $f( \Sigma_j)$  are different points in $\mathcal{M}$.
\end{enumerate}
\begin{prop}\label{smoothing2}
(Notations as above) Let
$\mathcal{E}_j=\overline{(\mathcal{C}-\sum_{1\le k \le
j}\mathcal{D}_k)}\subset \mathcal{C}$. Denote by
$\mathcal{D}_j\cap\mathcal{E}_j=\Xi_j$.

\noindent (1) There exists $d_j$, such that or any line bundle $L$ on $\mathcal{D}_j$ which is
a pull back of some line bundle on the coarse space $D_j$ with
degree at least $d_j$,
$$H^1(\mathcal{D}_j,L\otimes f|_{\mathcal{D}_j}^*T_\mathcal{M}(-\sum
\Sigma_i-\Xi_j))=0.$$

\noindent (2)  Let $d_j$ be as in (1).  Let $\eta$ be a general point on $Z$, if for any $j$,
 $$m_j\ge h^1(\mathcal{C},f^*T_{\mathcal{M}}(-\sum \Sigma_i))+d_j,$$
then we have
$H^1(\mathcal{T}_{\eta}, \phi_{\eta}^*T_\mathcal{M}(-\sum \Sigma_i))=0,$
where $\phi_{\eta}=\phi|_{\mathcal{T}_{\eta}}.$

\end{prop}
\begin{proof} (1) The push-forward of $f|_{\mathcal{D}_j}^*T_\mathcal{M}(-\sum \Sigma_i-\Xi_i)$
 to the coarse space $D_j\cong \P^1$ is a vector bundle on $\P^1$,
which decomposes as $\oplus_r \mathcal{O}(a^j_r)$ for a collection of integers
$a^j_r$. Then it follows from \eqref{orbicurve} that we can choose $d_j=\max_r\{-a^j_r\}$.

\noindent
(2) Because of (\ref{smoothing}.b) we know $\mathcal{T}_{\eta}$ is 1-isomorphic to $\mathcal{C}$. Below, we will use the same notations to denote the corresponding substacks of $\mathcal{T}_{\eta}$ and $\mathcal{C}$. Tensoring $f^*T_\mathcal{M}(-\sum \Sigma_i)$ with the exact sequence (set $\mathcal{E}_0=\mathcal{T}_{\eta}\cong\mathcal{C}$)
$$ 0\to \mathcal{O}_{\mathcal{D}_j}(-\Xi_j)\to \mathcal{O}_{\mathcal{E}_{j-1}}\to \mathcal{O}_{\mathcal{E}_{j}}\to 0. $$
Applying the argument of (\cite{kollarrc}, II.7.10), (\ref{smoothing}.c) and the assumption that
 $$m_j\ge h^1(\mathcal{C},f^*T_{\mathcal{M}}(-\sum \Sigma_j))+d_j,$$
 we know that for each component $\mathcal{D}_j\subset \mathcal{T}_{\eta}$ which is not mapped to a point,
$$H^1(\mathcal{D}_j,\phi_{\eta}|_{\mathcal{D}_j}^*T_\mathcal{M}(-\sum \Sigma_i-\Xi_j))=0.$$

By induction, to prove
$H^1(\mathcal{T}_{\eta}, \phi_{\eta}^*T_\mathcal{M}(-\sum \Sigma_i))=0,$
it suffices to verify that
$H^1(\mathcal{E}_s,\phi_{\eta}^*T_\mathcal{M}(-\sum \Sigma_i))=0$, where $\mathcal{E}_s$ is the subcurve consisting of all the components which are mapped constantly.  Hence,
$$H^1(\mathcal{E}_s,\phi_{\eta}^*T_\mathcal{M}(-\sum \Sigma_i))=H^1(\mathcal{E}_s,\mathcal{O}_{\mathcal{E}_s}(-\sum \Sigma_i))^{\oplus d}=H^1(E_s,\mathcal{O}_{E_s}(-\sum \sigma_i))^{\oplus d},$$
where $d$ is the dimension of $\mathcal{M}$ and
$(E_s,\sigma_i)$ is the coarse moduli space of $(\mathcal{E}_s,\Sigma_i)$.
Since each connected component of $E_s$ is  of arithmetic genus 0, 
it follows from (4) that
the restriction of $\mathcal{O}_{E_s}(\sum \sigma_i)$ to each component
is of degree at most 1, which implies that
$H^1(E_s,\mathcal{O}_{E_s}(-\sum \sigma_i))=0$.
\end{proof}
 Let us summarize our discussion by the following theorem.
\begin{thm}[Smoothing]\label{tool}
 Let $(f:\mathcal{C}\to \mathcal{M},\Sigma_i\subset \mathcal{C})$ be an $n$-pointed twisted stable map which satisfies (1)-(3) as in \eqref{construction} and (4) as the assumption above \eqref{smoothing2}. Then there exists another $n$-pointed twisted stable map  $(g:\mathcal{D}\to\mathcal{M},\Phi_i\subset \mathcal{D})$ of genus 0 of possibly larger degree such that
\begin{enumerate}
\item $\mathcal{D}$ is smooth;
\item for each $i$, $\Sigma_i$ and $\Phi_i$ are 1-isomorphic. Moreover, under this 1-isomorphism, $f(\Sigma_i)=g(\Phi_i)$; and
\item $H^1(\mathcal{D},g^*T_{\mathcal{M}}(-\sum \Phi_i))=0.$
\end{enumerate}
In particular, if  we assume that all the $\Sigma_i$'s are  contained in
the scheme locus of $\mathcal{C}$, then $\mathcal{D}\cong \P^1$.
\end{thm}
\begin{proof} For each $j$,  let $d_j$ be the constant we find in (\ref{smoothing2}.1).  In Construction \eqref{construction}, we take $m_j\ge h^1(\mathcal{C},f^*T_{\mathcal{M}}(-\sum \Sigma_i))+d_j$. Then it follows from (\ref{smoothing2}.2)
that the $n$-pointed twisted
stable map $(\phi_{\eta}:\mathcal{T}_{\eta}\to
\mathcal{M}, \Sigma_i)$ satisfies
$$H^1(\mathcal{T}_{\eta}, \phi_{\eta}^*T_\mathcal{M}(-\sum \Sigma_i))=0.$$
Since $\mathcal{T}_{\eta}$ can be written as a degeneration of a family $\mathcal{T}_S$
of smooth objects (see the remark after (\ref{versal})) over some base $S$,
we conclude that after a possible \'etale base change of $S$, $\phi_{\eta}$ can be
extended to a family of $n$-pointed twisted stable maps from $\mathcal{T}_S$ with
the restriction map on the $n$ marked gerbes being constant. We let a general
fiber of this family be our $\mathcal{D}$, thus it admits a representable morphism $g:\mathcal{D}\to \mathcal{M}$ with the marked gerbes $\Phi_i\cong \Sigma_i(1\le i \le n)$.  Then the vanishing of
$H^1(\mathcal{D},g^*T_{\mathcal{M}}(-\sum \Phi_i))$ is implied by the
upper semicontinuity. The last statement immediately follows from the fact that
$\mathcal{D}\setminus \cup_i{\Sigma_i}$ is a scheme.
\end{proof}

For varieties with only quotient singularities, we have the following result.
\begin{prop}(\cite{vi}, 2.8)\label{quot}
If $X$ a normal variety with quotient singularities, then it is the
moduli space of some smooth Deligne-Mumford stack $\mathcal{X}$. Furthermore, the coarse moduli map $\mathcal{X}\to X$ is isomorphic over the smooth locus $X^{sm}$ of $X$.
\end{prop}
With all these preparations, now we can show our main result of this section.

\begin{prop}\label{finitepoints}
Let $X$ be a projective variety with only isolated quotient
singularities and ${\rm dim}(X)\ge 2$. Assume its smooth locus $X^{sm}$ is rationally
connected and there are at most finitely many points in $X^{sm}$ which are
not in the very free locus $X^{vf}$ of $X^{sm}$. Then 
$X^{sm}$ is strongly rationally connected.
\end{prop}

\begin{proof}  We fix a general point $x\in X^{vf}$. We need to prove for any $y\in X^{sm}$, there exists a very free curve in $X^{sm}$ which passes through $y$. Consider the Hom space $\Hom (\P^1,X,0\mapsto x)$ and the universal morphism 
$$ev: \Hom (\P^1,X,0\mapsto x) \times \P^1 \to X.$$
From our assumption, we know that there is a smooth open subscheme $V\subset \Hom (\P^1,X, 0\mapsto x)$ parametrizing very free curves, such that $y$ is in the closure of $ev(V\times \P^1)$. This implies that there exists a connected smooth curve $B^0\to V$ such that $y\in \overline{ev(B^0\times \P^1)}$.  
 Thus after a possible base change  $D^0\to B^0$, we can assume that there is a section $\Sigma^0_1: D^0 \to  D^0 \times \P^1$, such that  $y$ is in the closure of  $\Sigma^0_1(D^0)$.

Let $\Sigma^0_2$ be the section given by $D^0\times \{0\}$, then $(D^0 \times \P^1 \to X,\Sigma^0_1,\Sigma^0_2)$ induces a morphism
$u_{D^0}:D^0\to\overline{\mathcal{M}}_{0,2}(\mathcal{X},d) $, where $\mathcal{X}$ is
the smooth stack as in (\ref{quot}) and $d$ is the degree of the parametrized curves. From the properness of
$\overline{\mathcal{M}}_{0,2}(\mathcal{X},d)$, after possibly taking a 
quasi-finite cover again, we can complete $u_{D^0}$ as a family of 2-pointed
twisted stable maps
$$(\pi:\mathcal{C}\to D, f_D:\mathcal{C}\to \mathcal{X}, \Sigma_1:D\to \mathcal{C}, \Sigma_2:D\to \mathcal{C})$$ with a limiting object 
$$(f_q: \mathcal{C}_q\to
\mathcal{X},\Sigma_1(q),\Sigma_2(q))$$ such that $f_q(\Sigma_1(q))=y, f_q(\Sigma_2(q))=x $ for
a point $q$ in $D$. 

We apply (\ref{tool}) to $(f_q: \mathcal{C}_q\to \mathcal{X},\Sigma_1(q),\Sigma_2(q))$, where $\mathcal{M}$ is $\mathcal{X}$ and $\mathcal{M}^0$ is $X^{vf}$ in (\ref{construction}). From (\ref{tool}), we conclude that there is a morphism $g:\P^1 \to \mathcal{X}$, such that $g(\P^1)$ passes through $x, y $ and  $H^1(\P^1, g^*T_{\mathcal{X}}(-\Sigma_1(q)-\Sigma_2(q)))=0,$ hence $g^*T_{\mathcal{X}}(-\Sigma_1(q))$ is semi-positive. Since for $s\not=\Sigma_1(q)$, the tangent map
$$T_{s,\P^1}\oplus H^0(\P^1, g^*T_{\mathcal{X}}(-\Sigma_1(q)))\to T_{g(s),\mathcal{X}}$$ 
of the universal morphism
$\P^1\times \underline{\Hom} (\P^1,\mathcal{X}, \Sigma_1(q)\mapsto y)\to \mathcal{X}$
at $(s,g)$ is given by  $dg(s)+\phi(s,g),$
where $$\phi(s,g): H^0(\P^1, g^*T_{\mathcal{X}}(-\Sigma_1(q)))\to T_{g(s),\mathcal{X}}$$
is the evaluation morphism (cf. \cite{kollarrc}, II.3.4).  Since $\phi(s,g)$ is surjective,
we conclude that the universal morphism is smooth at $(s,g)$ for any $s\not=\Sigma_1(q)$. In particular, it is equidimensional. So we can choose a general member $[H]\in \underline{\Hom} (\P^1,\mathcal{X},\Sigma_1(q)\mapsto y)$ such that $H(\P^1)$ does not pass through the finitely many stacky points in $\mathcal{X}$. Then its image on $X$ gives a very free rational curve passing through $y$ which is entirely in the smooth locus $X^{sm}$ of $X$.
\end{proof}

\section{$C$-noncontracted minimal model program}

 In this section, we finish the proof of \eqref{main}-\eqref{singularity}. We always assume that $S$ is a projective surface with at
worst quotient singularities. We denote by $S^{sm}$ its smooth locus
and $S^{vf}$ the (possibly empty) very free locus of $S^{sm}$. Since
surface $klt$ singularities precisely mean quotient singularities
(cf. \cite{kollarmori}, 4.7 and 4.8), for any surface $S$ with $klt$ singularities,
there exists a smooth proper Deligne-Mumford stack $\mathcal{S}$
which realizes $S$ as its coarse moduli space (cf. \eqref{quot}). To prove
(\ref{main}), we will run a specific procedure of minimal model program
(MMP) for $klt$ surface pair to verify the assumptions of
(\ref{finitepoints}).
\begin{defn}
 For a $klt$ pair $(S,\Delta)$ and a fixed irreducible curve $C$ in $S$, a {\it $C$-noncontracted MMP} is a MMP process for the pair $(S,\Delta)$ $$ p:(S, \Delta)=(S_0,\Delta_0)\to (S_1,\Delta_1)\to \cdots \to (S_n,\Delta_n),$$
such that in each step, we do not contract the image of $C$.
\end{defn}
Like the usual MMP, the procedure terminates. In the last step, we will have a surface pair $(S_n,\Delta_n)$ satisfying one of the following:
\begin{enumerate}
 \item[(i)] $(S_n,\Delta_n)$ is a log del Pezzo surface with $\rho(S_n)=1$, and $0\neq p_*[C]\in N_1(S_n)$;
\item[(ii)] $(S_n,\Delta_n)$ admits a Fano contraction to a curve $B$, such that $C$ is finite over $B$;
\item[(iii)] $(S_n,\Delta_n)$ is a minimal model; or
\item[(iv)] $\rho(S_n)\ge 2$ and $\overline{\NE}(S_n)$ has only one $(K_{S_n}+\Delta_n)$-negative extremal ray $R_{\ge0}[p_*C]$.
\end{enumerate}

The reason for us to consider $C$-noncontracted MMP is encrypted in the following proposition.

\begin{prop}\label{C-nonctr}
Assume $S^{sm}$ is rationally connected. If we run a $C$-noncontracted MMP for $(S,\Delta)$ and we end in case (i) or (ii), then $C\cap S^{vf}\neq\emptyset.$
\end{prop}
\begin{proof}
For each case, we will show that there is a free rational curve $h:\P^1 \to S_n^{sm}$,   such that $R=h(\P^1)$ has nonempty intersection with the image of $C$.
Assuming this for the moment, replacing $R$ by its general
deformation, we can think of $S\to S_n$ as an isomorphism over a
neighborhood of $R$. So the preimage of $R$ gives a free curve in
$S^{sm}$. Applying Hodge Index Theorem, we know that $R$ meets any very free curve in $S^{sm}$. It follows that $R\subset S^{vf}$ (cf. \cite{kollarrc}, IV.3.9.4).  Hence we conclude $C\cap S^{vf}\neq
\emptyset.$

To verify the claim: in Case(i), because of (\ref{rc}),
$S^{sm}$ contains a very free curve $R$. But $\rho(S)=1$, which
implies any two curves have a nonempty intersection. In Case(ii), we
can choose $R$ to be a general fiber.

\end{proof}

\begin{lemma}\label{mmp}
 If $p:(S,\Delta)\to(S',\Delta')$ is an extremal contraction of surface pairs and $(S,\Delta)$ is a log del Pezzo surface, then $(S',\Delta')$ is also a log del Pezzo surface.
\end{lemma}
\begin{proof}
If $\alpha \in \overline{\NE}(S')\setminus \{0\}$, then $p^*(\alpha)\in \overline{\NE}(S)\setminus \{0\}$. The lemma follows from
$$(K_{S'}+\Delta')\cdot \alpha=(K_{S}+\Delta)\cdot p^*(\alpha)<0.$$
\end{proof}

\begin{prop}\label{cdm1}
If $(S,\Delta)$ is a log del Pezzo surface, the complement of $S^{vf}$ in $S^{sm}$ consists of finitely many points.
\end{prop}

\begin{proof} If there is an irreducible curve in the complement, we take its closure in $S$, and denote it as $C$. Running a $C$-noncontracted MMP, since the Kodaira dimension  $\kappa(S,\Delta)=-\infty$, $(S,\Delta)$ does not have a minimal model. So we only need to rule out the case (iv). It follows from \eqref{mmp} that  $K_{S_n}+\Delta_n$ is negative on  $\overline{\NE}(S_n)\setminus \{ 0 \}$. If $\rho(S_n)\ge 2$, it implies that there exist more than one extremal rays. In particular, at least one of them is not $\mathbb{R}_{\ge 0}[p_*C]$. But this is a contradiction.
\end{proof}

\begin{proof}[Proof of \eqref{main}]: It  now follows from \eqref{finitepoints} and \eqref{cdm1}.
\end{proof}

\begin{proof}[Proof of \eqref{duval}]: We run a (regular) $K_S$-MMP with the last surface $S_n$ yielding a Fano contraction. For each step, since the MMP process $p_i:S_i\to S_{i+1}$ contracts the exceptional curve $E_i$ to a smooth point $P_i$, we conclude that $p_i(S_i^{sm})= S^{sm}_{i+1}$, which implies that $S_{i+1}^{sm}$ is rationally  connected. 
\begin{claim}
 $S^{sm}_n$ is strongly rationally connected. 
\end{claim}
 \begin{proof}
If $\rho(S_n)=1$, this follows from \eqref{main}. If $S_n$ is contracted to a curve, for any curve $C\subset S_n$, it is easy to see that we can always find a very free curve $R\subset S^{sm}_n$ such that $R\cdot C>0$. This implies $S^{vf}\cap C\neq\emptyset$ and the complement of $S^{vf}$ in  $S^{sm}$ is a set of finite points, then we can apply \eqref{finitepoints}.\end{proof}

We show \eqref{duval} by decreasing induction on $i$. Since $S_{i+1}^{sm}$ is  strongly rationally connected, we can find a very free rational curve $f:\P^1\to S_{i+1}^{sm}$ passing through $P_i$ such that $f$ is an isomorphism over a neighborhood of $P_i$ in $ f(\P^1)$. Let $f':\P^1 \to S_{i}$ be the lifting of $f$. 

Since $E_i$ contains at most one singularity which is of type $A_r$ (cf. \cite{km99}, 3.3), if we let $\pi:\mathcal{S}_i\to S$ be the smooth stacky cover as in \eqref{quot}, and $h:\P^1\to \P^1$ a degree $r$ cover totally ramified at $0\mapsto 0$, then we conclude that the composition $f'\circ h:\P^1\to S_i$ factors through a morphism $g: \P^1 \to \mathcal{S}_i$ (cf. \cite{km99}, 4.7 and 4.10).   Since $g(\P^1)\cap S^{vf}_i\neq \emptyset$, we can apply \eqref{tool} to it with two general points $\Sigma_1,\Sigma_2\in \P^1$.  Thus we obtain  a very free curve $g':\P^1\to \mathcal{S}_i $, whose image is contained in the scheme locus,  and $g'(\P^1)$ meets $E_i$ at its general points. Thus we conclude that 
$E_i\cap S^{vf}_i\neq \emptyset.$
\end{proof}

\noindent
{\it Proof of (\ref{singularity})}: Let us consider a resolution $g:Y\to S$. Then $Y$ is a smooth rational surface. We have
$H^i(Y,\mathcal{O}_Y)=0, \forall i>0.$
The Leray spectral sequence says
$H^i(S,Rg_*^j(\mathcal{O}_Y))\Rightarrow H^{i+j}(Y,\mathcal{O}_Y).$

Because $S$ is normal, we have
$$H^0(S,Rg_*^1\mathcal{O}_Y)\cong H^2(S,\mathcal{O}_S)\cong H^0(S,\omega_S).$$
\noindent Then it follows from the argument of (cf. \cite{kollarrc},
IV.3.8) that the rational connectedness of $S^0$
implies the vanishing of $H^0(S,\omega_S)$.

\section{A Gorenstein log del Pezzo surface in characteristic 2 }

 In this section, we present an example due to Koll\'ar,
which is a del Pezzo surface of degree 2 over $\F_2$ with $A_1$-singularities,
whose smooth locus does not contain any free rational curves, even after any field extension of $\F_2$. This means \eqref{rc} fails if we drop the assumption that the ground field has characteristic 0. 

 We start with $\P^2_{\F_2}$, it has seven $\F_2$-points.
There are also seven lines defined over $\F_2$, each of which passes through
precisely three $\F_2$-points. We blow up all these seven
$\F_2$-points to get a degree 2 {\it weak del Pezzo surface} $X$, i.e., $-K_X$ is big and nef. The birational transforms of these $\F_2$-lines are
$(-2)$-curves on $X$. After contracting all these $(-2)$-curves, we have a
Gorenstein log del Pezzo surface $S$ of degree 2, which contains seven
$A_1$-points. 

As any degree 2 Gorenstein log del Pezzo surface, $S$ is a degree 4 hypersurface in $\P(1,1,1,2)$ (cf. \cite{dol}, 8.6.1). More precisely, $|-K_X|$ is a base-point-free linear system which yields a generically degree 2 morphism $h:X\to \P^2$. Since $h$ contracts all the birational transforms of lines, we conclude that $h$ factors through a morphism $g:S\to \P^2$. This means that the usual presentation of  a smooth degree 2 del Pezzo surface as a double cover of $\P^2$ branched over a quartic curve can be generalized to $S$.

In more explicit terms, $S$ is given by an equation $x_3^2=f_4(x_0,x_1,x_2)$, where $f_4$ is
 a degree 4 homogeneous polynomial. Let $C\in |\mathcal{O}_{\P^2}(4)|$ be the quartic curve given by $f_4$. We know that $S$
 is singular at $(x_0,x_1,x_2,x_3)$ if and only if $C$ is singular at $(x_0,x_1,x_2)$. Moreover, the fact that $S$ has only $A_1$ singularities is equivalent to saying that $f$ only has non-degenerate critical points.  For more detailed explanation, see
 (\cite{kollarrc} V.5.6).

\begin{lemma} Let $I$ be the ideal sheaf of the seven
$A_1$ singularities on $S$.
 There exists an exact sequence oh the following form
$$0\to g^*I(1)\to \Omega^1_S\to g^*\mathcal{O}_{\P^2}(-2)\to 0.$$
\end{lemma}

\begin{proof}
This is essentially the exact sequence of (\cite{ko95}, 9.4). In fact, thanks to the Koll\'ar's construction in \cite{ko95}, we know that $f_4$ gives a morphism
 $df_4:\mathcal{O}_{\P^2}(-4) \to \Omega^1_{\P^2}$
and there exists an exact sequence
$$0\to g^*\mbox{coker}(\mathcal{O}_{\P^2}(-4) \to \Omega^1_{\P^2})\to \Omega^1_S\to g^*\mathcal{O}_{\P^2}(-2)\to 0.$$
(For the meaning of this exact sequence, see (\cite{ko95}, 9).) Choosing
local coordinates to compute, we know
$$(dx,dy)/(xdx+ydy) \cong (x,y)\frac{dx}{y},$$
Therefore, the cokernel of $df_4$ which is a rank one sheaf,  can be written as $I\otimes L$ for some line bundle $L$.  To determine $L$, we compute the first Chern class
$$c_1(\mbox{coker}(\mathcal{O}_{\P^2}(-4) \to \Omega^1_{\P^2}))=c_1(\mathcal{O}_{\P^2}(1)),$$
\noindent
thus $\mbox{coker}(\mathcal{O}_{\P^2}(-4) \to \Omega^1_{\P^2})\cong I(1)$ as $c_1(I)=0$.
\end{proof}

Now for any rational curve $f:\P^1\to S^{sm}$, pulling
back the exact sequence by $f^*$, we know that 
$$0\to  f^*\mathcal{O}_{\P^2}(1)\to f^*\Omega^1_S.$$
Hence it is not a free curve. 

\begin{q}
 It is interesting to ask whether there exists a smooth Fano variety which is defined over a field of characteristic $p>0$ and does not contain any (very) free curve.
\end{q}

%%%%%%%%%%%%%%%%%%%%%%%%%%%%%%%%%%%%%%%%%%%%%%%%%%%%%%%%%%%%%%%%%%%%%%%%

\noindent 2-380, Department of Mathematics, Massachusetts Institute of
Technology,\\
\noindent 77 Massachusetts Avenue Cambridge, MA 02139\\
\noindent {\it email} cyxu@math.mit.edu


\begin{thebibliography}{BDPP04}

\bibitem[Ar74]{ar}
Artin, M.; Versal deformations and algebraic stacks.
{\it Invent. Math.} {\bf 27} (1974), 165-189.

\bibitem[AV02]{av}
Abramovich, D.; Vistoli, A.; Compactifying the space of stable maps.
{\it J. Amer. Math. Soc.} {\bf  15}  (2002),  no.{\bf 1}, 27--75 (electronic).

\bibitem[Ca92]{Ca}
Campana, F.;
Connexit\'e rationelle des vari\'et\'es de Fano. {\it Ann. Sci. \'Ecole Norm. Sup.} {\bf 25} (1992) 539-545.

\bibitem[DM69]{dm}
Deligne, P.; Mumford, D.; The irreducibility of the space of curves of given genus. {\it Inst. Hautes \'E‰tudes Sci. Publ. Math.} No. {\bf 36} (1969), 75--109.

\bibitem[Dol10]{dol}
Dolgachev, I.; Topics in Classical Algebraic Geometry.  \url{http://www.math.lsa.umich.edu/~idolga/topics.pdf}.

\bibitem[GHS03]{ghs}
Graber, T.; Harris, J.; Starr, J.; Families of rationally connected varieties.  {\it J. Amer. Math. Soc. } {\bf 16}  (2003),  no. {\bf 1}, 57--67 (electronic).


\bibitem[HM07]{hm07}
Hacon, C.; McKernan, J.; Shokurov's Rational Connectedness Conjecture,
{\it Duke Math. J.} {\bf 138} (2007) no. {\bf 1}, 119-136.


\bibitem[HT08]{ht08}
Hassett, B.; Tschinkel, Y.; Approximation at Places of Bad Reduction for
Rationally Connected Varieties,
{\it Pure and Applied Mathematics Quarterly} {\bf 4} (2008) no. {\bf 3}, 743-766.

%\bibitem[HT09]{ht}
%Hassett, B.; Tschinkel, Y.; Weak approximation for hypersurfaces of low degree.
%{\it Algebraic geometry-Seattle 2005. Part 2}, 937-955, Proc. Sympos. Pure Math., {\bf 80}, Part {\bf 2}, {\it Amer. Math. Soc., Providence, RI}, 2009.


\bibitem[KM98]{kollarmori}
Koll\'ar, J.; Mori, S.; {\it Birational Geometry of Algebraic Varieties},
Cambridge Tract. in Math. {\bf 134}, Cambridge University Press, Cambridge, 1998.

\bibitem[KM99]{km99}
Keel, S.; M$^{\rm c}$Kernan, J. Rational curves on quasi-projective surfaces.
{\it Mem. Amer. Math. Soc. } {\bf 140}  (1999),  no. {\bf 669}.

\bibitem[KMM92a]{KMM}
Koll\'ar, J.; Miyaoka, Y.; Mori, S.;
Rational Connectedness and Boundedness of Fano Manifolds,
{\it J. Diff. Geom.} {\bf 36} (1992), no. {\bf 3}, 765--779.

\bibitem[KMM92b]{kmm}
Koll\'ar, J.; Miyaoka,Y.; Mori,S.;
Rationally connected varieties.
{\it J. Algebraic Geom.} {\bf 1} (1992), no. {\bf 3}, 429--448.


\bibitem[Kn08]{kn}
Knecht, A.;
Weak approximation for general degree 2 del Pezzo surfaces. arXiv: 0809.9261.

\bibitem[Ko95]{ko95}

Koll\'ar, J.;
Nonrational hypersurfaces.
{\it J. Amer. Math. Soc. } {\bf 8} (1995), no. {\bf 1}, 241--249.

\bibitem[Ko96]{kollarrc}
Koll\'ar, J.; {\it Rational curves on algebraic varieties.},
Ergeb. Math. Grenz. {\bf 3}. Folge, {\bf 32}.
Springer-Verlag, Berlin, 1996.


\bibitem[Ol06a]{ol2}
Olsson, M.; Deformation theory of representable morphisms of algebraic stacks.  {\it Math. Z.} {\bf 253}  (2006),  no. {\bf 1}, 25--62.

\bibitem[Ol06b]{ol3}
Olsson, M.; $\underline {\rm Hom}$-stacks and restriction of scalars. {\it Duke Math. J.} {\bf 134}  (2006),  no. {\bf 1}, 139--164.

\bibitem[Ol07]{ol}
Olsson, M.; On (log) twisted curves, {\it Comp. Math.} {\bf 143} (2007), 476-494.




\bibitem[Vi87]{vi}
Vistoli, A.;
Intersection theory on algebraic stacks and on their moduli spaces.
{\it Invent. Math.} {\bf 97} (1989), no. {\bf 3}, 613--670.

\bibitem[Xu09]{xu}
Xu, C.; Weak approximation for low degree del Pezzo surfaces.
arXiv:0902.1765.

\bibitem[Zh06]{zhang}
Zhang, Q.; Rational connectedness of log $\Q$-Fano varieties.
{\it J. Reine Angew. Math.} {\bf 590} (2006), 131--142.
\end{thebibliography}
\end{document}